\documentclass[a4paper]{article}
\usepackage[utf8]{inputenc}

\usepackage[margin=1.3in]{geometry}

\usepackage{amsmath,amssymb,amscd,amsthm,amsfonts}
\usepackage{graphicx,subcaption}
\usepackage{hyperref}
\usepackage{dsfont}
\usepackage{xcolor}
\usepackage{color,soul}
\usepackage[utf8]{inputenc}
\usepackage[nobysame, alphabetic]{amsrefs}
\usepackage[capitalise]{cleveref}
\usepackage{thm-restate}

\newtheorem{theorem}{Theorem}[section]
\newtheorem{corollary}[theorem]{Corollary}

\newtheorem{question}[theorem]{Question}

\theoremstyle{definition}

\newcommand{\R}{\mathds{R}}
\newcommand{\rr}{\mathds{R}}
\newcommand{\zz}{\mathds{Z}}

\title{Cookie cutters: Bisections with fixed shapes}

\author{Patrick Schnider \\
        Department of Mathematics and Computer Science,
        University of Basel\\
        Department of Computer Science, ETH Z\"{u}rich, Switzerland \\ {\tt patrick.schnider@inf.ethz.ch}
        \and
        Pablo Sober\'{o}n\thanks{The research of P. Sober\'on was supported by NSF CAREER award no. 2237324 and a PSC-CUNY Trad B award.} \\ Department of Mathematics, Baruch College, City University of New York, USA \\
        Department of Mathematics, The Graduate Center, City University of New York, USA \\ {\tt psoberon@gc.cuny.edu}}

\date{}

\begin{document}

\maketitle

\begin{abstract}
In a mass partition problem, we are interested in finding equitable partitions of smooth measures in $\rr^d$.  In this manuscript, we study the problem of finding simultaneous bisections of measures using scaled copies of a prescribed set $K$.  We distinguish the problem when we are allowed to use scaled and translated copies of $K$ and the problem when we are allowed to use scaled isometric copies of $K$.  These problems have only previously been studied if $K$ is a half-space or a Euclidean ball.  We obtain positive results for simultaneous bisection of any $d+1$ masses for star-shaped compact sets $K$ with non-empty interior, where the conditions on the problem depend on the smoothness of the boundary of $K$.  Additional proofs are included for particular instances of $K$, such as hypercubes and cylinders, answering positively a conjecture of Sober\'on and Takahashi.  The proof methods are topological and involve new Borsuk--Ulam-type theorems.

\end{abstract}

\section{Introduction}

Given finite measures or finite sets of points in a geometric space, finding a fair way to split the space into pieces is a natural goal, often called a mass partition problem \cite{Matousek2003, Kano2021, RoldanPensado2022}.  Fairness corresponds to the pieces having the same size in each measure, or having the same number of points of each set.  The quintessential example is the ham sandwich theorem \cite{Steinhaus1938, Stone:1942hu}, that states that \textit{given $d$ mass distributions in $\rr^d$, there exists a hyperplane that simultaneously halves each of them}.  In this paper, we study mass partition problems in which one of the pieces has a fixed shape.  As with many mass partition results, this instance has a food-related interpretation.

Assume you have a cookie dough with three ingredients, e.g., the dough, chocolate chips and coconut sprinkles. To also have fresh cookies tomorrow, you want to use half of the dough right now. You also want that the half you leave for tomorrow contains half of the dough, half of the chocolate chips, and half of the coconut sprinkles. Unfortunately, the dough is already rolled out, and the only thing you can still do to it without destroying the chocolate chips is scaling the dough. Further, the only cutting device you have at hand is a single cookie cutter. Is it always possible to scale the rolled out dough in such a way, that you can nicely bisect all ingredients with a single cut with the cookie cutter? In this paper, we will show that if the cookie cutter has a nice enough shape the answer to this is `yes', even in higher dimensions.  In some cases, we can even avoid rotating the cookie cutter.

In formal terms, we define a \emph{cookie cutter} in $\R^d$ as a compact subset of $\R^d$ with non-empty interior whose boundary contains at least one point at which it is smooth. We say that a cookie cutter is \emph{smooth} if its boundary is smooth everywhere. We also consider \emph{star-shaped} cookie cutters. Instead of ingredients, we want to bisect \emph{mass distributions}. A mass distribution $\mu$ on $\R^d$ is a Borel measure on $\R^d$ such that $0<\mu(\R^d)<\infty$ and absolutely continuous with respect to the Lebesgue measure. Instead of scaling the dough, we will think of scaling the cookie cutters. More formally, for a cookie cutter $C$, another cookie cutter $C'$ is a \emph{homothetic copy} of $C$ if $C'$ can be obtained by a scaling and translation of $C$.  We include limiting cases of these shapes as valid homothetic copies. 
Similarly, we say that a shape $C'$ is \emph{similar to} a shape $C$ if $C'$ can be obtained from $C$ by scaling, translation, and rotation.

The main goal of this paper is determine which families of measures can be bisected by similar or homothetic copies of a single set $C$.  This problem has only been solved when $C$ is a half-space (which is not a cookie cutter) or a sphere (which is a smooth star-shaped cookie cutter) \cite{Stone:1942hu}.

In Section \ref{sec:smooth_cookies}, we prove the first main result of this paper:

\begin{restatable}{theorem}{smoothcookies}\label{thm:smooth_cookies}
Let $\mu_0,\ldots,\mu_{d}$ be $d+1$ mass distributions on $\R^d$ and let $C$ be a smooth cookie cutter. Then there exists a homothetic copy $C'$ of $C$ such that $\mu_i(C')=\frac{1}{2}\mu_i(\R^d)$ for all $i\in\{0,\ldots,d\}$.
\end{restatable}

Note that if the boundary of $C$ is not smooth everywhere, then the analogous result does not hold: consider three point-like masses on the line $x=y$ in $\R^2$ and let $C$ be an axis-parallel square. No axis-parallel square can have all three points on its boundary. However, allowing rotations, a solution exists.  The result for $d$ mass distributions follows directly from the ham sandwich theorem \cite{Steinhaus1938, Stone:1942hu}, as we can obtain any half-space as the limiting shape of homothetic copies of $C$.

In a first step towards a more general result when allowing rotations in \cref{sec:hypercubes} we prove a similar result for hypercubes:

\begin{restatable}{theorem}{hypercubes}\label{thm:hypercubes}
 Let $\mu_0,\ldots,\mu_{d}$ be $d+1$ mass distributions in $\R^d$. Then there exists a hypercube $C$ such that $\mu_i(C)=\frac{1}{2}\mu_i(\R^d)$ for all $i\in\{0,\ldots,d\}$.
\end{restatable}

In the plane, this result states that any three mass distributions can be simultaneously bisected by a square, confirming a conjecture of Sober\'on and Takahashi \cite{Soberon2023}.  This was known for two measures if we include the additional condition that the square must be axis-parallel \cite{Uno:2009wk, Karasev:2016cn}. We give an alternative, simpler proof of \cref{thm:hypercubes} in the plane in \cref{sec:squares}, which generalizes to cylinders in higher dimensions.  This approach also gives a new proof of the Sober\'on--Takahashi theorem on equipartitions using pairs of parallel hyperplanes \cite{Soberon2023}.

Finally, in \cref{sec:non_smooth_cookies} we prove the second main result of the paper, which removes the assumption of smooth boundaries, but requires rotation and reflection of the cookie cutters:

\begin{restatable}{theorem}{nonsmoothcookies}\label{thm:non_smooth_cookies}
Let $\mu_0,\ldots,\mu_{d}$ be $d+1$ mass distributions on $\R^d$ and let $C$ be a (not necessarily smooth) cookie cutter. Then there exists a set $C'$ that is a similar copy of $C$ or a reflection of a similar copy of $C$ such that $\mu_i(C')=\frac{1}{2}\mu_i(\R^d)$ for all $i\in\{0,\ldots,d\}$.
\end{restatable}

As mentioned at the start of the introduction, many results about mass partitions related to this paper have been illustrated using food, such as cutting sets with few hyperplanes (pizza cuttings) \cite{Barba2019,Blagojevic2022,Hubard2020,Hubard2024,Schnider2021}, cutting sets into few pieces and then distributing them among players (cake cuttings) \cite{Steinhaus49} and cutting mass assignments on affine subspaces (fairy bread cutting) \cite{AxelrodFreed2022,Blagojevic2023,Camarena2024,Schnider:2020kk}.  Results related to our main theorems include studying which fractions of $d$ or $d+1$ mass distributions on $\rr^d$ can be cut simultaneously using a single convex set \cite{Aichholzer:2018gu,Akopyan:2013jt,Blagojevic:2007ij}, and the existence of simultaneous bisections of $d+1$ mass distributions on $\rr^d$ with wedges and cones \cite{Barany:2002tk,Schnider2019,Soberon2023}.

The proof methods are based on equivariant algebraic topology.  The tools we require are simple homotopy arguments and direct applications of Borsuk--Ulam type theorems that have elementary proofs (in a few instances, the Borsuk--Ulam theorem itself). For ease of presentation we first present all our proofs only for star-shaped cookie cutters. In Section \ref{sec:non_star} we then present the additional arguments required to adapt the proofs to the general statements.

\section{Bisections with cookie cutters}\label{sec:smooth_cookies}

Before proving Theorem \ref{thm:smooth_cookies} for star-shaped cookie cutters let us give a brief overview of the ideas. A homothetic copy of a star-shaped cookie cutter is uniquely defined by the location of the star point and a scaling factor. Given a mass distribution $\mu$, for every location of the star point there is (essentially) a unique scaling factor for which the resulting cookie cutter bisects $\mu$. Taking the difference of the mass outside and inside of the cookie cutter for other masses we get a function whose zeros correspond to simultaneous bisections. In our proofs we will show that we can extend this function to an antipodal map from the sphere $S^{d}$ to $\R^d$ without adding any zeros that do not correspond to simultaneous bisections. The result then follows from the Borsuk--Ulam theorem.

\begin{theorem}\label{thm:smooth_star}
Let $\mu_0,\ldots,\mu_{d}$ be $d+1$ mass distributions on $\R^d$ and let $C$ be a smooth star-shaped cookie cutter. Then there exists a homothetic copy $C'$ of $C$ such that $\mu_i(C')=\frac{1}{2}\mu_i(\R^d)$ for all $i\in\{0,\ldots,d\}$.
\end{theorem}

For convenience, we will use the following form of the Borsuk--Ulam theorem.

\begin{theorem}
    Let $B^d$ be the unit ball of dimension $d$.  Every continuous map $f:B^d \to \rr^d$ that is antipodal on the boundary of $B^d$ must have zero.
\end{theorem}

\begin{proof}[Proof of \cref{thm:smooth_star}]
We first discuss a way to represent the homothetic copies of the cookie cutter $C$ in $\R^d$. As $C$ is star-shaped, there is a point $p$ such that for every point $x\in C$ the segment $px$ is in $C$.  Given a point $c \in \rr^d$ and a scaling factor $s > 0$, we define the homothetic copy $C(c,s) = s(C-p)+c$

Given a point $c \in \rr^d$, for $s_1 \le s_2$ we have $C(c,s_1)\subseteq C(c,s_2)$.  Therefore, 
$\mu_{0}(C(c,s_1))\leq \mu_{0}(C(c,s_2))$.  The values of $s$ for which $\mu_{0}(C(c,s))=\mu_{0}(\R^d)/2$ form an interval with midpoint $s(c)$.

Consider the ball $B^d$.  For each $v \in B^d$, we will define a set $C(v)$.  For $\|v\|<1/2$, $C(v)$ will be a set of the form $C(c,s)$ for some $c \in \rr^d$ and $s>0$.  For $\|v\| \ge 1/2$, $C(v)$ will be a half-space.  For each point $v$ such that $\|v\|<1/2$, consider
\begin{align*}
c(v) & = \left(\frac{\|v\|}{2\|v\|-1}\right)v    \\
C(v) & = C\Big(c(v),s(c(v))\Big)
\end{align*}

Note that the interior of $(1/2)B^d$ parametrizes all points of $\rr^d$ with $c(v)$, so we are parametrizing all copies of $C$ that bisect $\mu_{0}$.  Note that $2\|v\|-1 < 0$.  In other words, we translate $C$ in the direction of $-v$ some amount and then scale the set to contain half of $\mu_0$.

For every direction $u \in S^{n-1}=\partial B^d$, let $m(u)$ be the point on the intersection of $\partial C(\bar{0},1)$ and the ray starting at $\bar{0}$ in direction $u$.  Let $H(u)$ be the supporting half-space of $C(\bar{0},1)$ on $m(u)$, that is, $\partial H(u)$ is the tangent hyperplane at $m(u)$ and $H(u)$ contains $\bar{0}$.   Denote by $n(u)$ be normal vector of $H(u)$, pointing towards $\bar{0}$.  Note that $n(u) \neq u$.  For $1/2 \le \|v\| \le 1$, let $\alpha = \|v\|$ and $u = v / \|v\|$.  Let

\[
n(v) = \frac{(2-2\alpha)n(u) + (1-2\alpha)u}{\|(2-2\alpha)n(u) + (1-2\alpha)u\|}
\]

This unit vector is interpolating between $n(u)$ when $\alpha = 1/2$ and $-u$ when $\alpha =1$.  The vector $n(v)$ is well defined since $n(u) - u \neq 0$, so the denominator is never zero.  Let $C(v)$ be the half-space whose normal vector is $n(u)=n(v/\|v\|)$, containing the side in direction $n(u)$, and bisecting $\mu_{0}$.  If there is an interval of such half-spaces, we pick the one at the midpoint.  Note that for $u \in S^{n-1}$, the half-spaces $C(u)$ and $C(-u)$ share the same boundary hyperplane (orthogonal to $u$) but point to opposite sides.  The behavior of $C(v)$ is described in Figure \ref{fig:homothetic}.  Now, we can define a function

\begin{figure}
    \centering
    \includegraphics[width = \textwidth]{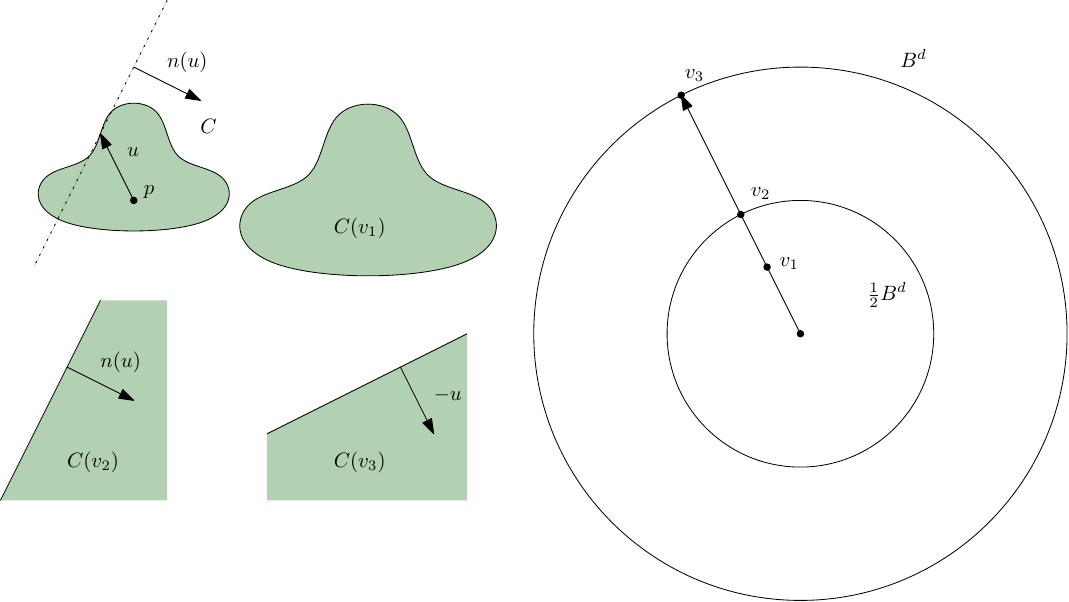}
    \caption{This figure describes how $C(v)$ changes as $v$ moves in $B^d$ towards the boundary.  Within $(1/2)B^d$, the set $C(v)$ is a homothetic translated copy of $C$.  When $v=v_2$, the magnitude of $v$ is $1/2$, and $C(v)$ is a half-space orthogonal to $n(u)$.  As we keep increasing the magnitude of $v$, we change the direction of the half-space $C(v)$ until it is orthogonal to $u$ (and points in the direction of $-u$.}
    \label{fig:homothetic}
\end{figure}

\begin{align*}
f: B^d & \to \rr^d \\
v & \mapsto (\mu_1(C(v))-\mu_1(\rr^d \setminus C(v)), \dots, \mu_d(C(v))-\mu_d(\rr^d \setminus C(v))).
\end{align*}

By construction, this map is continuous (the only delicate point is at the points $v$ such that $\|v\|=1/2$).  If $\|v\|\ge 1/2$, then $C(v)$ is a continuously moving half-space, with the topology induced by the affine oriented Grassmanian, so $f$ is continuous.  As $\|v\|$ approaches a point $v'$ such that $\|v'\|=1/2$ from the interior of $(1/2)B_d$, then the center $c$ goes to infinity in the direction of $v'$, so $\|c(v)\|, s(c(v))$ both tend to infinity.  Moreover, for any compact set $K$, $K \cap C(v)$ approaches $K \cap C(v')$ (using the Hausdorff metric).  Since the value of $\mu_{i}(C(v))$ can be checked using a sequence of increasing compact sets, the function $f$ is continuous at $v'$.

The function $f$ is also antipodal on the boundary of $B^d$.  By the Borsuk--Ulam theorem, it must have a zero, which corresponds to a homothetic copy of $C$ (possibly a half-space) that bisects all mass distributions.

\end{proof}

\section{Bisections with squares and cylinders}\label{sec:squares}

In this section we give a simple proof of a bisection theorem for cylinders.  We denote a cylinder in $\rr^d$ as the product of a $(d-1)$-dimensional ``flat ball'' in $\rr^d$ with an orthogonal segment.  During this section, let $\rho > 0$ be a fixed real number.  For the proof, we will use

\[
K = \left\{(x_1,\dots, x_{d-1},x_d) \in \rr^d: \sum_{i=1}^{d-1}x_i^2 \le 1, \ -\rho \le x_d \le \rho\right\}
\]

Changing the range of $x_d$ has no effect of the proof and the result below also follows.

\begin{theorem}\label{thm:cylinders}
    Let $d$ be a positive integer and $K$ as described above.  Let $\mu_0, \dots, \mu_{d}$ be $d+1$ mass distributions on $\rr^d$.  There exists a scaled isometric copy $K'$ of $K$ such that for all $i = 0,1,\dots, d$,
    \[
    \mu_i(K') = \frac{1}{2}\mu_i(\rr^d).
    \]
\end{theorem}

In the theorem above, we allow for ``infinite'' scalings of $K$, which translate to half-spaces (there are other scalings possible by enlarging $K$ from a non-smooth point on its boundary, but we won't need them).  In the case $d=2$ and $\rho=1$, the set $K$ is a square.  We therefore confirm a conjecture by Sober\'on and Takahashi.

\begin{corollary}
 Let $\mu_0,\mu_1,\mu_2$ be three mass distributions on $\R^2$. Then there exists a square $C$ such that $\mu_i(C)=\frac{1}{2}\mu_i(\R^2)$ for all $i\in\{1,2,3\}$.   
\end{corollary}

Of course, if we change the value of $\rho$, the corollary above works as well for rectangles of fixed aspect ratio.

\begin{proof}[Proof of \cref{thm:cylinders}]
    We prove a slightly stronger statement.  If we denote the $x_d$-axis as the direction of $K$, we naturally induce a direction line for any scaled isometric copy $K'$ of $K$.  The copy we search for will have its direction line going through the origin.

    For any direction $v \in S^{d-1}$, let $K(v)$ be the isometric copy of $v$ with direction line $\{\alpha v: \alpha \in \rr\}$.  Note that $K(v)=K(-v)$.  We now construct a map $f:S^{d-1} \times [0,1] \to \rr^d$.  We first define it on $S^{d-1}\times [0,1)$.

    Given $(v,\alpha) \in S^{d-1}\times [0,1)$ and $\beta>0$, consider the sets of the form
    \[
        K'(v,\alpha, \beta) = \left( \frac{\alpha}{1-\alpha}\right)v + \beta K(v).
    \]    
    The set of values $\beta$ such that $K(v,\alpha,\beta)$ contains exactly half of $\mu_{0}$ is an interval, so we can pick $\beta$ to be the midpoint of said interval to define a set $K'(v,\alpha)$.  For a fixed $v$, as $\alpha \to 1$ we have $\beta \to \infty$.  Let $H(v)$ be the translate of the half-space $\{x: \langle x,v\rangle \ge 0\}$ that contains exactly half of $\mu_0$ (as usual, if there is a range we pick the middle half-space).  For any compact set $R$, we have that $R\cap K'(v,\alpha) \to R \cap  H(v)$ as $\alpha \to 1$, under the Hausdorff metric.  For $\alpha =1$, we define $K'(v,\alpha) = H(v)$.  Finally, the map we want is
    \begin{align*}
        f: S^{d-1}\times [0,1] & \to \rr^d \\
        (v,\alpha) & \mapsto \left( \mu_1(K'(v,\alpha)) - \frac{1}{2}\mu_1(\rr^d),\dots, \mu_d(K'(v,\alpha)) - \frac{1}{2}\mu_d(\rr^d)\right)
    \end{align*}

    By construction, $f$ is continuous.  If $f(v,\alpha) = 0$, then $K'(v,\alpha)$ is the scaled isometric copy of $K$ we were looking for.  Let us assume that the map $f$ has no zeros and search for a contradiction.  We can do a standard dimension reduction argument and define
    \begin{align*}
        g: S^{d-1}\times [0,1] & \to S^{d-1} \\
        g(v,\alpha)& = \frac{f(v,\alpha)}{\|f(v,\alpha)\|}.
    \end{align*}

    Let $g_{\alpha}: S^{d-1}\to S^{d-1}$ be defined by $g_{\alpha}(v) = g(v,\alpha)$.  The map $g$ is an explicit homotpy between $g_0$ and $g_1$.  The map $g_0$ is even, since $K(v) = K(-v)$ implies $K'(v,0) = K'(-v,0)$ and therefore $g_0(-v) = g_0(v)$.  The map $g_1$ is odd, since $H(v)$ and $H(-v)$ are half-spaces with the same boundary hyperplane but different orientation, so $g_1(-v) = - g_1(v)$.  This means that the degree of $g_0$ is even, while the degree of $g_1$ is odd, contradicting the fact that there is a homotopy between them.  Therefore, the map $f$ must have a zero.
\end{proof}

\section{Bisections with hypercubes}\label{sec:hypercubes}
In this section we prove Theorem \ref{thm:hypercubes}. Using the symmetries of hypercubes we can show something slightly stronger: we will show that we can always find a bisecting hypercube which is either centered at the origin or for which the line through the origin and the center of the hypercube is orthogonal to one of its facets.

The main topological result, which our proof is based on is the following Borsuk-Ulam type theorem for Stiefel manifolds due to Chan, Chen, Frick and Hull \cite{FrickAndFriends} (an alternative proof can be found in \cite{Manta2024} or deduced from Fadell and Husseini's classic paper on their index \cite{Fadell:1988tm}).  We denote by $V_{d,k}$ the Stiefel manifold of all orthonormal $k$-frames in $\rr^d$.  We consider $\zz_2 = \{+1,-1\}$ with multiplication, and denote by $\varepsilon_j \in (\zz_2)^k$ the element that has $-1$ in the $j$-th coordinate and $+1$ elsewhere.

\begin{theorem}[\cite{FrickAndFriends}, Thm. 1.1]\label{Thm:FrickAndFriends}
    Let $1 \le k \le d$ be integers.  Every $(\zz_2)^k$-equivariant map 
    \[V_{d,k}\rightarrow\R^{d-1}\oplus\R^{d-2}\oplus\cdots\oplus\R^{d-k}\]
    has a zero.
    Here $\varepsilon_j$ acts non-trivially precisely on the $j$th factor $\R^{d-j}$ and by $(x_1,\ldots,x_j,\ldots,x_d)\mapsto (x_1,\ldots,-x_j,\ldots,x_d)$ on $V_{d,k}$.
\end{theorem}

For the case $k=d$ above, we consider $\rr^0 = \{0\}$.  Consider the sphere $S^d$ with its standard embedding in $\R^{d+1}$ and let $N$ be its (fixed) north pole (that is, the point with coordinates $(0,\ldots,0,1)$ in the standard embedding). We call the antipodal point of the north pole the south pole. We say that a $(k+1)$-frame of orthogonal unit vectors $(v_0,\ldots,v_d)\in V_{d+1,d+1}$ is \emph{north-facing} if $N$ lies on the plane spanned by $v_0$ and $v_d$.
Denote by $N_{d+1}$ the space of all north-facing $(d+1)$-frames. Note that reversing $v_0$ defines a $\zz_2$-action on $N_{d+1}$. We first prove the following

\begin{theorem}\label{Thm:NorthFace}
    Every $\zz_2$-equivariant map $f:N_{d+1}\rightarrow\R^d$ has a zero.
\end{theorem}

\begin{proof}
    We extend $f$ to a $(\zz_2)^{d+1}$-equivariant map
    \[g:V_{d+1,d+1}\rightarrow\R^{d}\oplus\R^{d-1}\oplus\cdots\oplus\R^{0}\]
    with the property that $p\in N_{d+1}$ gets mapped to $(f(p),0,\ldots,0)$ and all the points that get mapped to values of the form $(x,0,\ldots,0)$ are in $N_{d+1}$.
    The statement then follows from \Cref{Thm:FrickAndFriends}.

    We define $g$ by specifying the value of $x_j \in \rr^{d+1-j}$ for all $j$ and declaring $g(v_0,\dots, v_{d}) = (x_0,\dots, x_{d})$. To define $x_1$, we just smooth out $f$ towards $0$ in a sufficiently small neighborhood of $N_{d+1}$ and set $0$ everywhere else.  A simple way of doing this is choosing some $0<\varepsilon<1$.  Given $p=(v_0, \dots, v_d) \in V_{d+1,d+1}$ such that $N \not\in \operatorname{span}(v_0, v_d):=H$, denote by $q$ the closest point to $N$ in $H \cap S^{d}$.  If $\operatorname{dist}(q,N) = \tau \le \varepsilon$,  We can take the rotation in $\operatorname{span}(q,N)$ that takes $q$ to $N$ and extend it to $S^{d}$.  Let $p'=(v'_0, \dots, v'_d)$ be the image of $(v_0, \dots, v_d)$ under this rotation, which is in $N_{d+1}$.  Finally, we declare $x_1 = (1-\tau/\varepsilon)f(p')$.
    
    For all the other non-trivial coordinates, i.e., $x_i$ for $2 \le i \le d$, consider the hyperplane $h_i$ spanned by $\{v_0,\dots,v_{d}\}\setminus \{v_i\}$. The direction of $v_i$ defines a positive side of this hyperplane $h_i$. Let $d(h_i,N)$ denote the distance between $h_i$ and $N$. We declare
    \[
        x_i = \begin{cases}
            \Big( d(h_i,N), 0,\dots, 0\Big) \in \rr^{d+1-i} & \mbox{ if $N$ is on the positive side of $h_i$} \\
            \Big( -d(h_i,N), 0,\dots, 0\Big) \in \rr^{d+1-i} & \mbox{ otherwise}
        \end{cases}
    \]

    The function $g$ defined above is continuous and $\zz^{d+1}$-equivariant.  Let us analyze the zeros of $g$.  Note that $\bigcap_{i=2}^{d}h_i$ is the span of $v_0$ and $v_d$.  Therefore, a zero of $g$ must be in $N_{d+1}$.  In this case, the first entry of $g$ is $f(p)$, and it must be zero.
\end{proof}

We are now ready to prove Theorem \ref{thm:hypercubes}

\hypercubes*

\begin{proof}
Embed $\R^d$ into $\rr^{d+1}$ by mapping $x \mapsto (x,1)$.  We can identify $\rr^d$ with the northern hemisphere of $S^d$ using central projection from the center of $S^d$. Consider a north-facing frame $\bar{p}=(v_0,\ldots,v_d)\in N_{d+1}$.  Recall that the action of $\zz_2$ on $N_{d+1}$ is such that $-\bar{p} = (-v_0, v_1, \dots, v_d)$.  Under central projection, $v_0$ defines a point $p_0$ in $\R^d$ or on the sphere at infinity.  Let $\ell$ be the line $\operatorname{span}(v_0, v_d)\cap \rr^d$.  The vectors $v_1, \dots, v_{d-1}$ are orthogonal to $N$, so they represent an orthonormal frame in $\rr^{d}$.  Let $v'$ be an unit vector on $\ell$ (the ambiguity here won't affect the construction).  Consider the frame in $\rr^d$ formed by $v_1, \dots, v_{d-1}, v'$.  If $\langle v_0, N \rangle > 0$, consider $C(\bar{p})$ the hypercube centered at $p_0$ whose facets are orthogonal to the frame $v_1, \dots, v_{d-1}, v'$.  We scale $C(\bar{p})$ so that it contains exactly half of $\mu_0$ (as before, if there is a range of scalings that satisfy this condition, we pick the middle one).  If $v_0$ becomes orthogonal to $N$, then we can consider $p_0$ as a point at infinity in the direction of $v_0$.  In this case, we make $C(\bar{p})$ the half-space that contains exactly half of $\mu_0$, orthogonal to $v_0$, which contains the side on the direction of $v_0$.

Now we define a function
\begin{align*}
    f: N_{d+1} & \to \rr^d \\
    f_i(\bar{p}) & = \begin{cases}
  \mu_i(C(\bar{p}))-\mu_i(\R^d\setminus C(\bar{p}))  & \mbox{if }\langle v_0, N \rangle \ge 0 \\
  \mu_i(\R^d\setminus C(-\bar{p}))-\mu_i(C(-\bar{p})) & \mbox{if } \langle v_0, N \rangle \le 0.
  \end{cases}
\end{align*}

Note that the two cases agree when $\langle v_0, N \rangle = 0$, as $C(\bar{p})$ and $C(-\bar{p})$ are complementary half-spaces.  The function is continuous, so by \Cref{Thm:NorthFace} it has a zero.  The zeros of $f$ correspond to hypercubes that bisects all mass distributions. \end{proof}

\section{Bisections with non-smooth cookie cutters}
\label{sec:non_smooth_cookies}

In this section we prove Theorem \ref{thm:non_smooth_cookies} for star-shaped cookie cutters.

\begin{theorem}\label{thm:non-smooth_star}
Let $\mu_0,\ldots,\mu_{d}$ be $d+1$ mass distributions on $\R^d$ and let $C$ be a (not necessarily smooth) star-shaped cookie cutter. Then there exists a set $C'$ that is a similar copy of $C$ or a reflection of a similar copy of $C$ such that $\mu_i(C')=\frac{1}{2}\mu_i(\R^d)$ for all $i\in\{1,\ldots,d+1\}$.
\end{theorem}

The proof combines ideas from the proofs of Theorem \ref{thm:hypercubes} and of Theorem \ref{thm:smooth_star}. We again use Theorem \ref{Thm:NorthFace} and restrict ourselves to north-facing frames with $v_0$ on the northern hemisphere, each of which will define a unique cookie cutter bisecting the last mass. Similar to the proof of Theorem \ref{thm:smooth_star} we use a neighborhood around the equator in which we rotate the (degenerate) cookie cutters to give rise to antipodal functions on the equator.

\begin{proof}
We again first discuss how we represent the relevant copies of the cookie cutter $C$. Let again $c$ denote the star center of $C$. Further denote by $p$ a point on the boundary $\partial C$ at which $\partial C$ is smooth. Recall that such a point exists by our definition of cookie cutters. Again, we will use the set $N_{d+1}$ of north-facing orthonormal $(d+1)$-frames in $\rr^{d+1}$ to parametrize copies of $C$.

We start with a simpler parametrization.  A copy of $C$ in $\R^d$ is determined by a point $v_0$ with a $d$-frame $(v_1,\ldots,v_d)$ attached to it, and a scaling factor $s$. We denote this copy by $C(v_0,v_1,\ldots,v_{d},s)$. We further choose our $d$-frame in such a way that the ray from $v_0$ in direction $v_d$ intersects the boundary of the copy $C'$ in the point $p'$ corresponding to the point $p$ where the boundary is smooth.

Consider now the sphere $S^d$ and divide it into seven parts $A_N, A_S, B_N, B_S, C_N, C_S$, and $E$ defined as follows. We first define $A_N\subseteq B'_N\subseteq C'_N$ as spherical balls centered at the north pole $N$ of radii $r_1<r_2<r_3$, where $r_3$ is smaller than the radius of the northern hemisphere. Now, let $B_N:=B'_N\setminus A_N$ and $C_N:=C'_N\setminus B'_N$. Note that $A_N$ is homeomorphic to a ball whereas $B_N$ and $C_N$ are homeomorphic to cylinders. Let $A_S$, $B_S$ and $C_S$ be the antipodal copies of $A_N$, $B_N$ and $C_N$ respectively. Finally, let $E$ be the remaining part of the sphere $S^d$ and note that $E$ is again a neighborhood of the equator. In the following we define a map $f:N_{d+1}\rightarrow \R^d$ depending on which part of $S^d$ the vector $v_0$ lies.

If $v_0$ lies in $A_N$, then we proceed as in the proof of Theorem \ref{thm:smooth_star}: the frame $(v_1,\ldots,v_d)$ defines the orientation of the cookie cutter, and we choose the unique scaling factor $s$ for which the cookie cutter $C^*:=C(v_0,v_1,\ldots,v_d,s)$ bisects the mass $\mu_{d+1}$. We again define
\[f_i(v_0,\ldots,v_d):=\mu_i(C^*)-\mu_i(\R^d\setminus C^*)\]
and $f=(f_1,\ldots,f_d)$. So far, this defines a function when $v_0$ is in $A_N$.

In order to define the function when $v_0$ is in $B_N$ we note that the intersection $A_N\cap B_N$ is homeomorphic to a $(d-1)$-dimensional sphere and that $N_{d+1}$ with $v_0$ restricted to this intersection, and thus also in $B_N$, splits into two connected components: one in which $v_d$ points to $N$, call it $B_N^+$, and one in which $v_d$ points away from $N$, call it $B_N^-$. We extend the function on both parts separately. The main idea of this is illustrated in Figure \ref{fig:non_smooth_rotation}. For $v_0\in B_N^+$ we take the same definition as for $v_0\in A_N$. For $v_0\in B_N^-$ we adapt the copies of the cookie cutter as follows: consider the (oriented) plane $\Pi$ spanned by the vectors $v_{d-1}$ and $v_d$. Using the space spanned by all other vectors as the axis of rotation, we can continuously rotate the cookie cutter in counter-clockwise direction with respect to the oriented plane $\Pi$ in such a way that at the boundary between $B_N$ and $C_N$ the cookie cutters with $v_0\in B_N^-$ are such that $p$ (the point at which $\partial C$ is smooth) lies on the ray from $c$ with direction $-v_d$. For each rotated copy we again choose the scaling factor so that the mass $\mu_{d+1}$ is bisected and define $f_i(v_0,\ldots,v_d):=\mu_i(C^*)-\mu_i(\R^d\setminus C^*)$.

\begin{figure}
\centering
\includegraphics[scale=0.6]{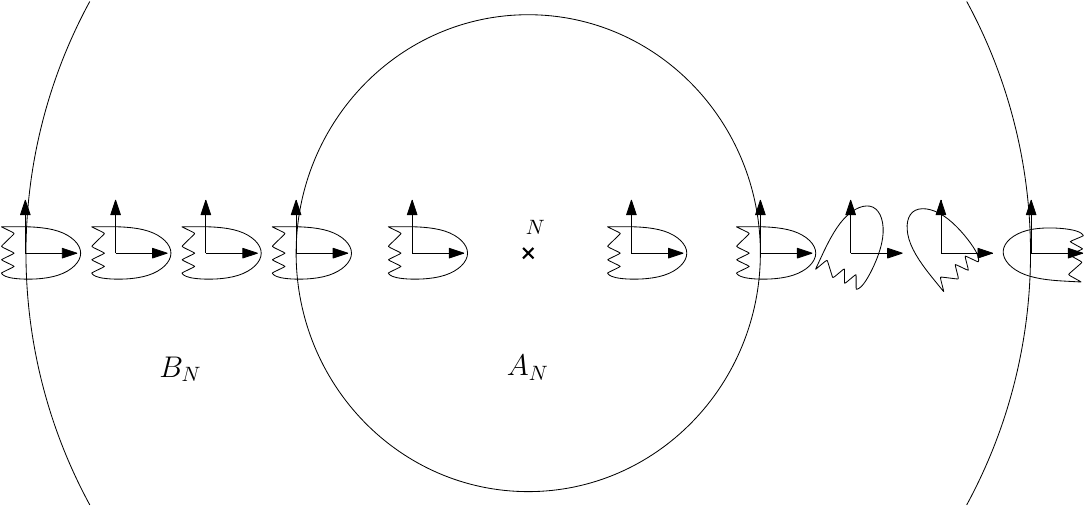}
\caption{In illustration of the rotation of the cookie cutters.}
\label{fig:non_smooth_rotation}
\end{figure}

If $v_0$ is in $C_N$, we again define $f$ in the same way, except that now in one connected component of $N_{d+1}$ restricted to $v_0\in C_N$ we take the rotated cookie cutters inherited from the rotation process in $B_N$. Considering the boundary between $C_N$ and $E$ to be the sphere at infinity, we now get that all the considered cookie cutters degenerate to hyperplanes, as the ray from $c$ to $N$ by construction passes through $p$. Finally, in $E$ we rotate these hyperplanes so that at the equator they are orthogonal to the line from $c$ to $N$, just as in the proof of Theorem \ref{thm:smooth_star}.
As before, we define $f_i(v_0,\ldots,v_d):=\mu_i(C^*)-\mu_i(\R^d\setminus C^*)$.

We have thus defined $f$ whenever $v_0$ lies in the northern hemisphere. It follows from the construction that $f$ is continuous. Further, at the equator we have $f(-v_0,v_1,\ldots,v_d)=-f(v_0,\ldots,v_d)$: the two considered hyperplanes both bisect $\mu_{d+1}$ and are thus the same, but the side that corresponds to the interior of the cookie cutter is different for both of them. Thus, the constructed function respects the required antipodality at the equator and we can symmetrically extend it to the southern hemisphere.

Combining all of the above, we get a $\zz_2$-equivariant map $f:N_{d+1}\rightarrow \R^d$, any zero of which corresponds to a simultaneously bisecting copy of the cookie cutter of $C$. The existence of a zero of $f$ now again follows from Theorem \ref{Thm:NorthFace}.
\end{proof}

\section{Generalizing to non-star-shaped cookie cutters}\label{sec:non_star}

In our arguments above we only used the fact that the cookie cutters are star-shaped in one step: arguing that there is a unique scaling that bisects the last mass. The idea of this section is to relax this condition, by showing that all the arguments still work as long as the relevant scalings are the zeroes of an odd continuous function, which they are.

To this end, we prove the following Borsuk-Ulam-type result:

\begin{theorem}\label{thm:function_bu}
Let $M$ be a manifold with a $\mathbb{Z}_2$-action ``$-$'' and assume that every $\mathbb{Z}_2$-map $M\rightarrow \mathbb{R}^d$ has a zero. Let $E=M\times [-1, 1]$ and let $f: E\rightarrow \mathbb{R}$ be a continuous map with the following properties:
\begin{itemize}
    \item[(i)] $f(x,-1)=-1$ and $f(x,1)=1$ for all $x\in M$;
    \item[(ii)] $f(-x,-t)=f(x,t)$ for all $x\in M$ and $t\in [-1,1]$;
    \item[(iii)] $f(x,t)\in (-1,1)$ for all $x\in M$ and $t\in (-1,1)$.
\end{itemize}
Let $Z=\{(x,t)\in E\mid f(x,t)=0\}$. Then every $\mathbb{Z}_2$-map $Z\rightarrow \mathbb{R}^{d}$ has a zero.
\end{theorem}

Before proving this, let us briefly explain how we can adapt the proofs above to work with this theorem. For readers interested in more details, we give a full proof of Theorem \ref{thm:smooth_cookies} at the end of the section. Recall that in the proofs of both Theorem \ref{thm:smooth_star} and Theorem \ref{thm:non-smooth_star} we described our candidate cookie cutters by the location of the star point and a scaling factor. The locations of the star point were parameterized as points on a manifold $M$, which is $S^d$ for Theorem \ref{thm:smooth_star} and $N_{d+1}$ for Theorem \ref{thm:non-smooth_star}. We can still do the same in the general setting, picking any point relative to the cookie cutter from which we scale radially. The scaling factor can be adapted to lie in the interval $[-1, 1]$. Looking at the last mass $\mu_{d+1}$ we get a map $f: M\times [-1, 1] \rightarrow \mathbb{R}$, which we can again normalize to map to $[-1, 1]$. It follows from the parametrizations in the two relevant proofs that this function $f$ satisfies the three properties required for Theorem \ref{thm:function_bu}. Further, the zeros $Z$ of this function correspond exactly to the cookie cutters bisecting $\mu_{d+1}$.

As in the proofs above we now define the $\mathbb{Z}_2$-maps $g_i: Z \rightarrow \mathbb{R}$ which for each cookie cutter $C$ in $Z$ is defined as $\mu_i(C)- \mu_i(\R^d\setminus C)$. Together these maps define a $\mathbb{Z}_2$-map $g: Z\rightarrow \R^d$, which by Theorem \ref{thm:function_bu} must have a zero. By construction, such a zero now corresponds to a cookie cutter which simultaneously bisects all masses.

It remains to prove Theorem \ref{thm:function_bu}.

\begin{proof}[Proof of Theorem \ref{thm:function_bu}]
Assume for the sake of contradiction that there is a $\mathbb{Z}_2$-map $g: Z\rightarrow \mathbb{R}^{d}$ which does not have a zero. We use this map to construct another $\mathbb{Z}_2$-map $g^*: M\rightarrow \mathbb{R}^{d}$, which is a contradiction.

To this end, we first extend $g$ to all of $E$ as follows: for some point $x\in M$, consider $f_x: [-1, 1]\rightarrow\R$, i.e., the function $f$ restricted to $x\in M$.  Formally, $f_x(t) = f(x,t)$. Let $t_1,\ldots,t_k$ be the zeros of $f_x$. For $t$ in the interval $[t_i, t_{i+1}]$ we define $h(x,t)$ as a linear interpolation between $g(t_i)$ and $g(t_{i+1})$, that is, $h(x,t):=\frac{t_{i+1}-t}{t_{i+1}-t_i}g(t_i)+\frac{t-t_i}{t_{i+1}-t_i}g(t_{i+1})$. For $t<t_1$ we define $h(x,t)$ as a linear interpolation between $0$ and $g(t_1)$ and similarly for $t>t_k$ we define $h(x,t)$ as a linear interpolation between $g(t_k)$ and $0$. Doing this for each $x\in M$ defines a continuous function $h: E\rightarrow \R^d$ which restricted to $Z$ is $g$. Further, from the antipodality of $g$ and $f$ it follows that $h(-x,-t)=-h(x,t)$.

Consider now the map $h^*: E\rightarrow \R^{d+1}$ defined by $h^*(x,t):=(h(x,t), f(x,t))$. Note that this map is is a $\mathbb{Z}_2$-map, that is, $h^*(-x,-t)=-h^*(x,t)$. Further, it has no zeros, as by construction of $h$ we have that $h(x,t)\neq 0$ if $f(x,t)=0$. Thus by normalizing we get a $\mathbb{Z}_2$-map $h': E\rightarrow S^d$. Note that by construction of $h$ and by condition (i) of $f$ we have $f(x,\pm 1)=(0,\ldots,0,\pm 1)$ for all $x\in M$ and further by condition (iii) of $f$ no other points of $E$ get mapped to $(0,\ldots,0,\pm 1)$. Thus by contracting $E$ at $t=-1$ and $t=1$ the map $h'$ extends to a $\mathbb{Z}_2$-map $h'': \Sigma E\rightarrow S^d$, where $\Sigma E$ denotes the suspension of $E$. As noted above, the suspension vertices are the only ones getting mapped to $(0,\ldots,0,\pm 1)$, so $h''$ restricted to $t=0$ is homotopic to a $\mathbb{Z}_2$-map $M\rightarrow S^{d-1}$, which is a contradiction.
\end{proof}

For the sake of illustration, we now give a full proof of Theorem \ref{thm:smooth_cookies}.

\smoothcookies*

\begin{proof}[Proof of Theorem \ref{thm:smooth_cookies}]
We first discuss a way to represent the homothetic copies of the cookie cutter $C$ in $\R^d$. Pick any point $p$ in the interior of $C$. Given a point $c \in \rr^d$ and a scaling factor $s > 0$, we define the homothetic copy $C(c,s) = s(C-p)+c$.

Assume without loss of generality that $\mu_0(\R^d)=1-\varepsilon$. For any point $c\in \R^d$ define the function $g(c,s):= \mu_0(C(c,s))-\mu_0(\R^d\setminus C(c,s))$. Note that $g(c,0)=-1+\varepsilon$ and $g(c,\infty)=1-\varepsilon$. Now let $\varphi$ be any homeomorphism $[0,\infty]\rightarrow [-1+\delta, 1-\delta]$. We thus get a function $g'(c,x):=g(c,\varphi^{-1}(x))$ with $g'(c,-1+\delta)=1-\varepsilon$ and $g'(c,1-\delta)=1-\delta$. This function can be extended to a function $g':\R^d\times[-1,1]\rightarrow[-1,1]$ with $g'(c,-1)=-1$ and $g'(c,1)=1$. Further, it follows from the construction that $g(c,x)\in (-1,1)$ for all $x\in(-1,1)$.

Consider the ball $B^d$.  For each $v \in B_d$ and $x\in[-1,1]$, we will define a set $C(v,x)$.  for $\|v\|<1/2$, $C(v,x)$ will just be the cookie cutter $C(c,\varphi^{-1}(x))$ for some $c \in \rr^d$.  For $\|v\| \ge 1/2$, $C(v)$ will be a half-space. We do this analogously to the proof of Theorem \ref{thm:smooth_star}, except that we consider all scaled cookie cutters, and not just the ones bisecting $\mu_0$. We can thus extend our function $g'$ defined above to all of $B^d$. Then, for $c\in\partial B^d$ we get that $g'(-c,-x)=g'(x,c)$. Extending this antipodally to $S^d$ we thus get the function $g'$ satisfying the properties of Theorem \ref{thm:function_bu}.

We have now parametrized cookie cutters as $S^d\times [-1,1]$. We define the map $f:S^d\times [-1,1]\rightarrow\R^d$ by $f_i(v,x):=\mu_i(C(v,x))-\mu_i(\R^d\setminus C(v,x))$. It now follows from Theorem \ref{thm:function_bu} that this map has a zero for which additionally $g'(v,x)=0$. By construction of the maps $f$ and $g'$ this corresponds to a cookie cutter that simultaneously bisects all mass distributions.
\end{proof}

\section{Conclusion}
 We have shown that for a large family of compact sets $C$ we can simultaneously bisect any $d+1$ mass distributions in $\R^d$ with similar copies of $C$. This opens a variety of follow-up questions. The first one is, whether the number of bisected mass distributions can be improved. For convex sets the answer is no: consider $d+2$ essentially point-like mass distributions where one of them, say $\mu_0$, lies inside the convex hull of the others. Now any convex set that simultaneously bisects the remaining mass distributions contains all of $\mu_0$. It is however possible, that for some sets we can bisect more than $d+1$ mass distributions.

\begin{question}
Is there a compact (maybe even star-shaped) set $C\subset\R^d$ such that any $d+2$ mass distributions in $\R^d$ can be simultaneously bisected with a similar copy of $C$? 
\end{question}

For similar copies we allow rotation, scaling and translation. It is clear that both translation and scaling are needed in general if we wish to bisect more than one mass distribution. However, rotations were only necessary for bisections with non-smooth cookie cutters. It is thus natural to wonder for which shapes scaling and translation is enough.

\begin{question}
For which sets $C\subset\R^d$ can any $d+1$ mass distributions in $\R^d$ be simultaneously bisected with a homothetic copy of $C$? 
\end{question}

Finally, there are the related algorithmic questions. As all our proofs are topological, they do not translate into any algorithm. The proof of Theorem \ref{thm:cylinders} is based on a degree argument, and such types of arguments can sometimes be adapted to give efficient algorithms, see e.g., \cite{Bereg2005,Pilz2021}.

\begin{question}
Given three point sets $P_1,P_2,P_3$ in the plane, how fast can we find a square which simultaneously bisects them?
\end{question}

Another approach that can lead to algorithmic results is based on the following idea: start by placing the point sets in so-called well-separated position and show that there is a unique bisector defined by one point of each class. Then continuously move the points to their correct positions, keeping track of the valid solutions, all of which are uniquely defined by one point of each class, by showing that they always appear or disappear in pairs, ensuring that the number of bisectors is always odd. This approach works for bisections with several lines \cite{Schnider2021} as well as with parallel hyperplanes \cite{Hubard2024}. However, already in the setting above this does not work immediately, as given three points in the plane there are generally infinitely many squares with these three points on the boundary. 


\end{document}